\documentclass[a4paper,11pt]{amsart}
\usepackage{amsmath, amssymb, amsthm}
\usepackage{verbatim}
\usepackage{hyperref}

\newcounter{theorem}
\newtheorem{thm}[theorem]{Theorem}
\newtheorem{lemma}[theorem]{Lemma}
\newtheorem{prop}[theorem]{Proposition}
\newtheorem{cor}[theorem]{Corollary}
\newtheorem{defn}[theorem]{Definition}

\theoremstyle{remark}
\newtheorem*{remark*}{Remark}

\numberwithin{equation}{section}
\numberwithin{theorem}{section}

\newcommand{\e}{\epsilon}
\newcommand{\dl}{\delta}

\newcommand{\R}{\mathbb{R}}

\renewcommand{\setminus}{\backslash}
\newcommand{\K}{\mathcal{K}}
\newcommand{\tens}{\otimes}
\newcommand{\dsum}{\oplus}
\newcommand{\bigdsum}{\bigoplus}
\newcommand{\dunion}{\amalg}

\newcommand{\Cu}{\mathcal{C}u}

\newcommand{\labelledthing}[2]{\hspace{4pt}\buildrel {#2} \over #1 \hspace{3pt}} 
\newcommand{\labelledleftarrow}{\labelledthing{\longleftarrow}}
\newcommand{\labelledrightarrow}{\labelledthing{\longrightarrow}}

\newcommand{\ccite}[2]{\cite[#1]{#2}}
\newcommand{\alabel}{\label}

\newcommand{\eb}{\partial_e}
\newcommand{\Lsc}{\mathrm{Lsc}}
\newcommand{\LAff}{\mathrm{LAff}}
\newcommand{\Aff}{\mathrm{Aff}}
\newcommand{\Prim}{\mathrm{Prim}}

\newcommand{\rod}{\mathrm{rod}}
\newcommand{\rank}{\mathrm{Rank}\,}

\title[On the structure of the Cuntz semigroup]{On the structure of the Cuntz semigroup in (possibly) nonunital $\mathrm{C}^*$-algebras}

\author{Aaron Tikuisis}
\address{\hskip-\parindent
Aaron Tikuisis, Mathematisches Institut der WWU M\"unster, Einsteinstra\ss{}e 62, 48149, M\"unster, Germany.}
\email{a.tikuisis@uni-muenster.de}
\author{Andrew Toms}
\address{\hskip-\parindent
Andrew Toms, Department of Mathematics, Purdue University, 150 North University Street, West Lafayette, IN 47907, USA.}
\email{atoms@purdue.edu}

\thanks{
The first-named author was supported by DFG (SFB 878).
The second-named author was supported by NSF grant DMS-0969246 and the 2011 AMS Centennial Fellowship.
The first-named author would also like to thank the Department of Mathematics at Purdue University for hosting him, during which time part of this article was written.
}


\begin{document}

\begin{abstract}
We examine the ranks of operators in semi-finite $\mathrm{C}^*$-algebras as measured by their densely defined lower semicontinuous traces.  We first prove that a unital simple $\mathrm{C}^*$-algebra whose extreme tracial boundary is nonempty and finite contains positive operators of every possible rank, independent of the property of strict comparison.  We then turn to nonunital simple algebras and establish criteria that imply that the Cuntz semigroup is recovered functorially from the Murray-von Neumann semigroup and the space of densely defined lower semicontinuous traces.  Finally, we prove that these criteria are satisfied by not-necessarily-unital approximately subhomogeneous algebras of slow dimension growth.
\end{abstract}

\maketitle

\section{Introduction}
It has recently become apparent that the question of which ranks, suitably defined, can occur in a simple and stably finite $\mathrm{C}^*$-algebra has considerable bearing on the deeper structure of the algebra.  The most significant example is due to Winter, who uses the notion of approximate divisibility of ranks in such algebras as an essential ingredient in his proof of $\mathcal{Z}$-stability for a wide class of nuclear $\mathrm{C}^*$-algebras \cite{Winter:pure}.  $\mathcal{Z}$-stability, in turn, is by now an indispensable tool in the effort to classify simple separable amenable $\mathrm{C}^*$-algebras via $\mathrm{K}$-theoretic data.

Several articles have appeared which concern ranks of operators in unital $\mathrm{C}^*$-algebras (\cite{BrownPereraToms,DadarlatToms:ranks,PereraToms:recasting,Toms:rigidity}); much of that work required further assumptions on the comparability of positive operators in the sense of Cuntz.  Here, we pursue two lines of research.  On the one hand, we prove that a unital simple $\mathrm{C}^*$-algebra with finitely many extreme tracial states contains positive operators of every possible rank, regardless of separability/nuclearity/comparability of positive operators.  We also begin to treat the nonunital and potentially stably projectionless cases, particularly nonunital approximately subhomogeneous algebras.  We begin by defining a measure of how closely a linear strictly positive function on the trace space of an algebra can be approximated by the rank function of a positive operator, and then show that this quantity is lower semicontinuous with respect to inductive limits.  After computing this invariant for recursive subhomogeneous algebras, we are able to prove that simple approximately subhomogeneous algebras with slow dimension growth have a Cuntz semigroup consisting of the Murray-von Neumann semigroup together with the lower semicontinuous strictly positive linear functions on the trace space.
In \cite{UnitlessZ}, the first named author shows that this structure for the Cuntz semigroup entails $\mathcal{Z}$-stability for the approximately subhomogeneous algebra;  in fact, slow dimension growth and $\mathcal{Z}$-stability are equivalent for simple approximately subhomogeneous algebras, extending the main result of \cite{Toms:rigidity} to the nonunital case. 

Our result for unital simple $\mathrm{C}^*$-algebras with finitely many extreme traces has some relationship to the main result of \cite{MatuiSato:Comp}, which says that for unital, simple, separable, nuclear $\mathrm{C}^*$-algebras with finitely many extreme traces, strict comparison and $\mathcal{Z}$-stability are equivalent.
On the one hand, their result implies ours under the additional hypotheses of separability, nuclearity, and strict comparison.
In \cite{Winter:pure}, one finds another $\mathcal{Z}$-stability theorem for $\mathrm{C}^*$-algebras under a number of hypotheses including (tracial) $m$-divisibility, a condition related to the ranks of positive operators.
Our result gives evidence that one may not be surprised that the $\mathcal{Z}$-stability theorem in \cite{MatuiSato:Comp} does not have a condition about the ranks of positive operators as a hypothesis: such a condition holds automatically for the algebras considered there.

The organization of this paper is as follows.
After preliminaries in Section \ref{PrelimSect}, we discuss traces and $\mathrm{C}^*$-algebras with compact primitive ideal space in Section \ref{TraceSect}.
Our first main result, concerning $\mathrm{C}^*$-algebras with finitely many extreme traces, is in Section \ref{FiniteTSect}.
The radius of divisibility, and its pertinent properties, are established in Section \ref{RODSect}.
Section \ref{CuCompSect} contains a computation of the Cuntz semigroup for simple exact $\mathrm{C}^*$-algebras, assuming that the Cuntz semigroup enjoys certain regularity properties.
Finally, we apply the theory developed in Sections \ref{RODSect} and \ref{CuCompSect} to approximately subhomogeneous $\mathrm{C}^*$-algebras in Section \ref{ASHSect}.

\section{Preliminaries} 
\alabel{PrelimSect}

Let $A$ be a $\mathrm{C}^*$-algebra. Let us consider on
$(A\otimes\mathcal K)_+$ the relation $a\precsim b$ if $v_nbv_n^*\to a$ for some sequence $(v_n)$ in $A\otimes\mathcal K$.
Let us write $a\sim b$ if $a\precsim b$ and $b\precsim a$. In this case we say that $a$ is Cuntz equivalent to $b$.
Let $\Cu(A)$ denote the set $(A\otimes \mathcal K)_+/\sim$ of Cuntz equivalence classes. We use $[ a ]$ to denote the class of $a$ in $Cu(A)$.  It is clear that
$[ a ]\leq [ b ] \Leftrightarrow a\precsim b$ defines an order on $Cu(A)$. We also endow $Cu(A)$
with an addition operation by setting $[ a ]+[ b ]:=[ a'+b' ]$, where
$a'$ and $b'$ are orthogonal and Cuntz equivalent to $a$ and $b$ respectively (the choice of $a'$ and $b'$
does not affect the Cuntz class of their sum).

We shall use $T(A)$ to denote the set of densely finite (a.k.a.\ densely defined) traces, as defined in \ccite{Definition 5.2.1}{Pedersen:CstarBook}.
Given $\tau \in T(A)$ we define a map
$d_\tau\colon  \Cu(A) \to [0,\infty]$ by the following formula:
\[
d_\tau([a]) := \lim_{n \to \infty} \tau(a^{1/n}).
\]
This is well-defined.
We could make this definition whenever $\tau$ is a $2$-quasitrace, but we wish to avoid defining these here.
Indeed, we are only concerned here with exact $\mathrm{C}^*$-algebras, and all $2$-quasitraces on an exact $\mathrm{C}^*$-algebra are traces; this was shown in \cite{BlanchardKirchberg:pi} by reducing to the unital case, which was proven in an unpublished manuscript of Haagerup \cite{Haagerup:quasitraces}.
All functionals on $\Cu(A)$ (suitably-defined) are of the form $d_\tau$ for a $2$-quasitrace $\tau$ \cite{BlackadarHandelman}.

Define $\iota:(A \tens \K)_+ \to \Lsc(T(A)), [0,\infty])$ by
\[ \iota(a)(\tau) = d_\tau(a). \]
Then $\iota(a)$ is lower semicontinuous (\ccite{Proposition 2.10}{PereraToms:recasting} in the unital case, \ccite{Section 5.1}{ElliottRobertSantiago} in general).
When we ask which ranks of positive operators occur, we mean: what is the range of $\iota$?

We shall say that $\Cu(A)$ is almost unperforated if, whenever $[x],[y] \in \Cu(A)$ satisfy
\[ (k+1)[x] \leq k[y] \]
for some $k \in \mathbb{N}$, it follows that $[x] \leq [y]$.

\section{Compact primitive ideal space and traces}
\alabel{TraceSect}

\begin{prop}
\alabel{CompactPrim}
Let $A$ be a $\mathrm{C}^*$-algebra.  The following statements are equivalent:
\begin{enumerate}
\item \alabel{CompactPrim-cpt} $\Prim(A)$ is compact;
\item \alabel{CompactPrim-cutdown} there exists $e \in A_+$ and $\e > 0$ such that $(e - \e)_+$ is full;
\item \alabel{CompactPrim-ped} there exists a full element in the Pedersen ideal of $A$;
\item \alabel{CompactPrim-cptcontain} there exist $a,b \in \mathrm{M}_n(A)_+$, some $n$, such that $a$ is full and $[a] \ll [b]$;
\end{enumerate}
\end{prop}

\begin{proof}
The equivalence of \ref{CompactPrim-cpt} and \ref{CompactPrim-cutdown} follows directly from \ccite{Proposition 3.5}{ProjlessReg}.
\ref{CompactPrim-cutdown} $\Rightarrow$ \ref{CompactPrim-ped} is clear since $(e-\e)_+$ is in the Pedersen ideal of $A$ (this is evident from the description of the Pedersen ideal in the proof of \ccite{Theorem 5.6.1}{Pedersen:CstarBook}).

\ref{CompactPrim-ped} $\Rightarrow$ \ref{CompactPrim-cptcontain}:
Let $a$ be full and in the Pedersen ideal.
We shall show that there exists $b \in M_n(A)_+$ for some $n$ such that $[a] \ll [b]$.

The proof of \ccite{Theorem 5.6.1}{Pedersen:CstarBook} shows that there exist $x_1,\dots,x_n \in A_+$ and $f_1,\dots,f_k \in C_c((0,\infty))_+$ such that
\[ a \leq \sum_{i=1}^n f_i(x_i). \]
Thus,
\[ [a] \leq \sum_i [f_i(x_i)] \ll \sum_i [x_i], \]
which is to say that if
\[ b := \bigdsum_i x_i \]
then $[a] \ll [b]$.

\ref{CompactPrim-cptcontain} $\Rightarrow$ \ref{CompactPrim-cpt}:
Given $a,b$ as in (iv), there exists $\epsilon>0$ such that $[a] \ll [(b-\epsilon)_+]$.  In particular $a \precsim (b-\epsilon)_+$, whence $(b-\epsilon)_+$ is full.  \ref{CompactPrim-cutdown} now follows for $\mathrm{M}_n(A)$ by setting $b=e$.  From the equivalence of \ref{CompactPrim-cpt} and \ref{CompactPrim-cutdown} we have that $\mathrm{Prim}(\mathrm{M}_n(A)) \cong \mathrm{Prim}(A)$ is compact, as required. 
\end{proof}

\begin{remark*}
It was pointed out by George Elliott that, the equivalence of \ref{CompactPrim-ped} and \ref{CompactPrim-cptcontain} can be generalized to the following fact: for any $\mathrm{C}^*$-algebra $A$, the Pedersen ideal of $A$ is
\begin{equation}
\{a \in A: \exists\ b \in (A \tens \K)_+\text{ s.t.\ }[a] \ll [b] \text{ in } \Cu(A)\}.
\alabel{PedCuDesc}
\end{equation}
Indeed, one inclusion is evident from the proof of \ccite{Theorem 5.6.1}{Pedersen:CstarBook}, while the other is shown by showing that the set \eqref{PedCuDesc} is an ideal.
This turns the equivalence of \ref{CompactPrim-ped} and \ref{CompactPrim-cptcontain} into a more general statement (and provides an alternate proof).
\end{remark*}

\begin{lemma}\label{full}
Let $A$ be a $\mathrm{C}^*$-algebra.  If $a \in A_+$ is full, then $\infty[a] := \sup_n n[a]$ is the largest element of $Cu(A)$.
\end{lemma}

\begin{proof}
By Brown's Theorem (\cite{Brown:StableIsomorphism}), we have $\overline{aAa} \otimes \mathcal{K} \cong A \otimes \mathcal{K}$; by identifying these, we may assume that $a$ is strictly positive.
Let $(p_n)$ be an increasing sequence of finite rank projections converging to $\mathbf{1} \in B(\mathcal{H})$ in the strong operator topology, so that $f_n := a^{1/n} \otimes p_n$ is an approximate unit for $A \otimes \mathcal{K}$.  Let $b \in (A \otimes \mathcal{K})_+$ and $\epsilon>0$ be given, and find $k \in \mathbb{N}$ large enough that 
\[
\|b^{1/2}f_kb^{1/2}-b\|< \epsilon.
\]
It follows by \ccite{Proposition 2.2}{Rordam:UHFII} that there is $x \in A \otimes \mathcal{K}$ such that 
\[
(xb^{1/2})f_k(xb^{1/2})^* = (b-\epsilon)_+, 
\]
whence
\[
\infty[a] \geq  \mathrm{rank}(p_k)[a] = [f_k] \geq [(b-\epsilon)_+].
\]
Since $\epsilon$ was arbitrary we have $\infty[a] \geq [b]$, as required.
\end{proof}

\begin{lemma}\label{densefinite}
Let $A$ be a $\mathrm{C}^*$-algebra with $a \in A_+$ full.  If $\tau$ is a lower semicontinuous trace on $A_+$ with $\tau(a)<\infty$, then $\tau$ is densely finite.
\end{lemma}

\begin{proof}
It will suffice to prove that $\tau((b-\epsilon)_+)<\infty$ for each $b \in A_+$ and $\epsilon>0$.  Let $b$ and $\epsilon$ be given.  By Lemma \ref{full} we have $\infty[a] \geq [b]$, so that 
\[
\infty[a] \gg [(b-\epsilon/2)_+].  
\]
It follows that 
\[
n[a] \geq [(b-\epsilon/2)_+]
\]
for some $n \in \mathbb{N}$, so we can find $x \in A \otimes \mathcal{K}$ such that 
\[
x\left( \oplus_{i=1}^n a \right)x^* = (b-\epsilon)_+.
\]
Extending $\tau$ to $A \otimes \mathcal{K}$ we then have
\begin{eqnarray*}
\tau((b-\epsilon)_+) & = & \tau((\oplus_{i=1}^n a^{1/2})(x^*x)(\oplus_{i=1}^n a^{1/2})) \\
& \leq & \|x\|^2 \tau(\oplus_{i=1}^n a) \\
& < & \infty.
\end{eqnarray*} 
\end{proof}

For any $a \in A$, set
\[ T_{a \mapsto 1}(A) := \{\tau \in T(A): \tau(a) = 1\}. \]
(In the case that $A$ is unital and $a=1$, this is of course the set of normalized traces.)

\begin{prop}
\alabel{TChoquet}
Suppose that $e \in A_+$ is full and in the Pedersen ideal of $A$.
It follows that $T_{e \mapsto 1}(A)$ is:
\begin{enumerate}
\item a base for the cone of densely finite traces on $A$; 
\item compact, in the topology of pointwise converge on the Pedersen ideal of $A$;
\item a Choquet simplex.
\end{enumerate}
\end{prop}

\begin{proof}
(i) 
On the other hand, suppose $\tau \in T(A) \setminus \{0\}$.
Then $\tau(e) > 0$, or else $e$ wouldn't be full.
Also, $\tau(e) < \infty$ since $e$ is in the Pedersen ideal.
Hence $\tau \in \R_+ T_{e \mapsto 1}(A)$, as required.

(ii)
We shall show that $T(A)$ is closed in the topology defined in \ccite{Section 3.2}{ElliottRobertSantiago}, from which it follows that it is compact in that topology.  By (i) and \ccite{Proposition 3.10}{ElliottRobertSantiago}, the restriction of this topology to $T_{e \mapsto 1}(A)$ agrees with the topology of pointwise convergence on the Pedersen ideal.

Suppose that $(\tau_i) \subseteq T_{e \mapsto 1}(A)$ is a net which converges to $\tau \in T(A)$.
Since $e$ is in the Pedersen ideal, by the proof of \ccite{Theorem 5.6.1}{Pedersen:CstarBook}, let $e \leq (a-\e)_+$ for some $a \in (A \tens \K)_+$ and $\e > 0$.
Since $[e]$ is full there is $n \in \mathbb{N}$ such that $[(a-\epsilon/2)_+] \ll n[e]$;  we can even arrange that $[(a-\epsilon/2)_+] \ll n[(e-\delta)_+]$ for sufficiently small $\delta$.  
We can then find $\e > 0$ and $x \in \mathrm{M}_n(A)$ such that
\[ 
e \leq (a-\epsilon)_+ = x \left( \oplus_{i=1}^n (e-\delta)_+ \right)x^* . 
\]
Therefore, for any $\eta \in T(A)$, we have
\[ 
\eta(e) \leq n \|x\|^2  \eta((e-\delta)_+) \leq K \eta((e-\delta)_+), 
\]
where $K = n \|x\|^2$.
Using the definition of the topology in \ccite{Section 3.2}{ElliottRobertSantiago}, we have
\begin{align*}
1/K &= \limsup \tau_i(e)/K \\
&\leq \tau_i((e-\delta)_+) \\
&\leq \tau(e) \\
&\leq \liminf \tau_i(e) \\
&= 1.
\end{align*}
Therefore, $\tau$ is a nonzero, densely finite trace on $A$.
It follows from \ccite{Proposition 3.10}{ElliottRobertSantiago}, that since $e$ is in the Pedersen ideal,
\[ \tau(e) = \lim \tau_i(e) = 1. \]
Hence, $\tau \in T_{e \mapsto 1}(A)$ as required.

(iii)
The cone of densely finite traces is, by \ccite{Corollary 3.3}{Pedersen:MeasureI} and \ccite{Theorem 3.1}{Pedersen:MeasureIII} a lattice cone.
It follows from this and (i),(ii) that $T_{e \mapsto 1}(A)$ is a Choquet simplex.
\end{proof}

\section{$\mathrm{C}^*$-algebras with finitely many extreme traces}
\alabel{FiniteTSect}

Suppose that the Pedersen ideal of $A$ contains a full element $e$, which we also ask to be positive.
We may clearly identify $\Lsc(T(A), (0,\infty])$ with the space 
\[ \LAff(T_{e \mapsto 1}(A), (0,\infty]) \]
of lower semicontinuous affine functions $T_{e \mapsto 1}(A)$ to $(0,\infty]$.

\begin{prop}
\alabel{ContinuousDenseSuffices}
Suppose that $A$ has a full positive element $e$ in its Pedersen ideal.
Then, the range of $\iota$ contains all of $\Lsc(T(A),(0,\infty])$ if and only if the range of $\iota$ contains a uniformly dense subset of continuous affine functions $T_{e \mapsto 1}(A) \to (0,\infty)$.
Moreover, in this case, $\Lsc(T(A),(0,\infty]) \subset \iota(\Cu(A) \setminus V(A))$.
\end{prop}

\begin{proof}
$\Rightarrow$ is obvious.
Conversely, suppose that $\iota$ contains a uniformly dense subset of continuous affine functions as above.
By \ccite{Theorem 11.8}{Goodearl:book}, every function in $\Lsc(T(A),(0,\infty])$ is the supremum of continuous linear functions $T(A) \to (0,\infty)$ (strictly speaking, \ccite{Theorem 11.8}{Goodearl:book} deals with functions whose codomain is $\R$, but the same proof works for codomain $(0,\infty]$).
By \cite{Edwards}, we may in fact obtain each function in $\Lsc(T(A),(0,\infty])$ as an increasing net of continuous linear functions.
Since $A$ is separable, $T(A)$ is metrizable and we can in fact replace such a net by a sequence.
The proof of this last statement doesn't quite go as one might expect, so we separate the argument as its own lemma.

\begin{lemma}
Let $X$ be a metrizable compact Hausdorff space.
Suppose that $f:X \to [0,\infty]$ is a lower semicontinuous function which is the pointwise supremum of an increasing net $(f_\alpha)$ of lower semicontinuous functions.
Then $f$ is the pointwise supremum of an increasing sequence $(f_{\alpha_i})$.
\end{lemma}

\begin{proof}
Let $(q_k)_{k=1}^\infty$ be a dense sequence in $[0,\infty)$.
For each $k$, $f^{-1}((q_k,\infty])$ is open, and since $X$ is metrizable, it is $\sigma$-compact.
Therefore, we can find an increasing sequence of open sets $(U_{k,i})_{i=1}^\infty$, each of which has compact closure, with union $f^{-1}((q_k,\infty])$; moreover $\overline{U_{k,i}} \subseteq f^{-1}((q_k,\infty])$ for each $i$.

By using the compactness of $\overline{U_{k,i}}$, lower semicontinuity of each $f_\alpha$ and the fact that the net $(f_\alpha)$ is increasing, we can find $\alpha_i$ such that $f_{\alpha_i}(x) > q_k$ for all $x \in \overline{U_{k,i}}$, $i=1,\dots,k$.
This condition, together with density of $\{q_k\}$, forces $f$ to be the pointwise supremum of $(f_{\alpha_i})$.
As the net $(f_\alpha)$ is increasing, it is clear that we can arrange that $(f_{\alpha_i})$ is increasing.
\end{proof}

We now have that for $f \in \Lsc(T(A), (0,\infty])$, there exists an increasing sequence $(f_n)$ of continuous linear functions $T(A) \to (0,\infty)$ whose pointwise supremum is $f$.
By hypothesis, let $S$ be a uniformly dense subset of continuous affine functions from $T_{e \mapsto 1}(A)$ to $(0,\infty))$ that is contained in the range of $\iota$.
By compactness of $T_{e \mapsto 1}(A)$, let $\dl_1 \in (0,1)$ such that $f_1(\tau) > \dl_1$ for all $\tau \in T_{e \mapsto 1}(A)$.
Then there exists $g_1 \in S$ such that
\[ f_1(\tau) - \dl_1 < g_1(\tau) < f_1(\tau) \]
for all $\tau \in T_{e \mapsto 1}(A)$.
Since $f_2(\tau) \geq f_1(\tau) > g_1(\tau)$, we may likewise pick $\dl_2 \in (0,1/2)$ such that
\[ f_2(\tau) - \dl_2 > g_1(\tau) \]
for all $\tau \in T_{e \mapsto 1}(A)$.
Then there exists $g_2 \in S$ such that
\[ f_2(\tau) - g_1(\tau) - \dl_2 < g_2(\tau) < f_2(\tau) - g_1(\tau) \]
for all $\tau \in T_{e \mapsto 1}(A)$.
That is to say,
\[ f_2 - \dl_2 < g_1 + g_2 < f_2. \]
We may continue this process, finding a sequence of numbers $\dl_n \in (0,1/n)$ and functions $g_n \in S$ such that
\[ f_n - \dl_n < g_1 + \cdots + g_n < f_n. \]
Evidently, for all $\tau \in T_{e \mapsto 1}(A)$, we have
\[ f(\tau) = \sup f_n(\tau) = \sum_{n=1}^\infty g_n(\tau). \]
Since $g_n \in S$, we have $g_n = \iota(a_n)$ for some $a_n \in (A \tens \K)_+$ with $\|a_n\|<1/2^n$.
Then we have
\[ \sum_{n=1}^\infty g_n = \iota( \dsum_{n=1}^\infty a_n), \]
as required.

The last statement of the proposition follows since $[\dsum_{n=1}^\infty a_n]$ is the supremum of the strictly increasing sequence $[\dsum_{n=1}^N a_n]$.
\end{proof}

\begin{prop}
\alabel{TotallyDisconnected}
Let $A$ be unital, such that $\eb T(A)$ is compact and totally disconnected.
Then every continuous function $\eb T(A) \to (0,\infty)$ is the uniform limit of functions in the range of $\iota$.
\end{prop}

\begin{proof}
Let $f: \eb T(A) \to (0,\infty)$ be a continuous function and let $\e > 0$.
By \ccite{Lemma 4.1}{DadarlatToms:ranks}, for each point $\tau \in \eb T(A)$, there exists $a_\tau \in (A \tens \K)_+$ and a neighbourhood $U_\tau$ of $\tau$ such that
\[ |d_\gamma(a_\tau) - f(\gamma)| < \e/3 \]
for all $\gamma \in U_\tau$.
By continuity of $f$, we may, by possibly shrinking $U_\tau$, also assume that $|f(\gamma) - f(\tau)| < \e/3$ for $\gamma \in U_\tau$.
Also, by the hypothesis that $\eb T(A)$ is completely disconnected, we may assume that $U_\tau$ is closed.

By compactness of $\eb T(A)$, let $U_{\tau_1},\dots,U_{\tau_n}$ be a finite subcover.
By shrinking some of the sets, we may in fact assume that $U_{\tau_1},\dots,U_{\tau_n}$ are pairwise disjoint.

Using \ccite{Lemma 4.5}{DadarlatToms:ranks} with $a_{\tau_i}$, let $b_i \in (A \tens \K)_+$ be such that
\[ d_\tau(b_i) = a_{\tau_i}(\tau) \]
for all $\tau \in U_{\tau_i}$, and
\[ d_\tau(b_i) < \frac{\e}{3(n-1)} \]
for $\tau \not\in U_{\tau_i}$.

Set $a = b_1 \dsum \cdots \dsum b_n$.
Then for $\tau \in \eb T(A)$, let $i$ be such that $\tau \in U_{\tau_i}$.
Then
\begin{align*}
&\ \ \ \ |d_\tau(b) - f(\tau)| \\
&\leq \sum_{j \neq i} |d_\tau(b_j)| + |d_\tau(b_i) - d_\tau(a_i)| + |d_\tau(a_i) - f(\tau_i)| + |f(\tau_i) - f(\tau)| \\
&< (n-1) \frac{\e}{3(n-1)} + 0 + \e/3 + \e/3 = \e,
\end{align*}
as required.
\end{proof}

\begin{cor}
If $A$ is unital and simple and $\partial_e T(A)$ is finite then $\iota$ is onto.
\end{cor}

\begin{proof}
This follows from Propositions \ref{ContinuousDenseSuffices} and \ref{TotallyDisconnected}.
\end{proof}

\section{The radius of divisibility}
\alabel{RODSect}

Abstracting the technique in the proof of \ccite{Theorem 3.4}{Toms:rigidity}, we introduce an invariant of the Cuntz semigroup called the radius of divisibility.
This name is inspired by the fact (shown in Proposition \ref{RankDimROD}) that it shares roughly the same relationship to the matricial-to-topological dimension of a recursive subhomogeneous algebra as does the radius of comparison, as defined in \ccite{Definition 4.1}{Toms:comparison}.
In fact, much as the radius of comparison has been used to show that simple approximately subhomogeneous algebras with slow dimension growth have strict comparison in their Cuntz semigroups, the radius of divisibility will be used in Corollary \ref{SDGranks} to show that the Cuntz semigroups of such algebras are also almost divisible.

Although we can phrase the following definition for general $e$, it is probably only useful in the case that $e$ is full and in the Pedersen ideal of $A$.

\begin{defn}
Let $A$ be a $\mathrm{C}^*$-algebra and let $e \in A_+$.
The \textbf{radius of density} of $A$ with respect to $e$ is the infimum of real numbers $r > 0$ such that, for any continuous linear function $f:T(A) \to (0,\infty)$, there exists $a \in (A \tens \K)_+$ such that for all $\tau \in T(A)$,
\[ |d_\tau(a) - f(\tau)| \leq r d_\tau(e). \]
We denote this quantity by $\rod(A,e)$.
\end{defn}

For a Choquet simplex $C$, let us denote by $\Aff(C)$ the set of continuous affine maps $C$ to $\R$, and by $\Aff(C)_{++}$ the subset of $\Aff(C)$ whose range is contained in $(0,\infty)$.

\begin{lemma}
\alabel{AffDensity}
Let
\[ C_1 \labelledleftarrow{\phi_2^1} C_2 \labelledleftarrow{\phi_3^2} \cdots \]
be an inverse sequence of Choquet simplices whose inverse limit is $C$.
Then
\begin{enumerate}
\item $\bigcup_{i=1}^\infty \Aff(C_i) \circ \phi_\infty^i$ is uniformly dense in $\Aff(C)$; and
\item $\bigcup_{i=1}^\infty \Aff(C_i)_{++} \circ \phi_\infty^i$ is uniformly dense in $\Aff(C)_{++}$.
\end{enumerate}
\end{lemma}

\begin{proof}
(i) is well-known. 

(ii): This is essentially contained in the proof of \ccite{Theorem 3.4}{Toms:rigidity}; however, for clarity, we will provide an explicit proof here.
By (i), it suffices to show that if $f \in \Aff(C_i) \circ \phi_\infty^i \cap \Aff(C)_{++}$ then $f \in \Aff(C_j)_{++} \circ \phi_\infty^j$ for some $j$.
Let $g \in \Aff(C_i)$ be such that $f= g \circ \phi_\infty^i$; we will in fact show that $g \circ \phi_j^i \in \Aff(C_j)_{++}$ for some $j \geq i$.
For a contradiction, suppose that this is false; that is, that for each $j \geq i$, there exists $x_j \in C_j$ such that $g(\phi_j^i(x_j)) \leq 0$.
Then for $j \geq i$, set
\[ \gamma_j = (\gamma_j^{(k)})_{k \geq i} := (\phi_j^i(x_j),\phi_j^{i+1}(x_j),\dots,\phi_j^{j-1}(x_j),x_j,x_{j+1},x_{j+2},\dots) \in \prod_{k=i}^\infty C_k. \]
By compactness of $\prod_{k=i}^\infty C_k$, let $\gamma = (\gamma^{(k)}) \in \prod_{k=i}^\infty C_k$ be a cluster point of the sequence $(\gamma_j)$.
Since $\gamma_j^{(k+1)} \circ \phi_{k+1}^k = \gamma_j^{(k)}$ for all $k \leq j$, it follows that $\gamma^{(k+1)} \circ \phi_{k+1}^k = \gamma^{(k)}$ for all $k$; that is, $(\gamma^{(k)})$ defines a point $\gamma$ in $C$.
However,
\[ f(\gamma) = g(\phi_\infty^i(\gamma)) \leq \limsup g(\phi_j^i(x_j)) \leq 0; \]
which is a contradiction.
\end{proof}

\begin{prop}
\alabel{IndLimitROD}
Let
\[ A_1 \labelledrightarrow{\phi_1^2} A_2 \labelledrightarrow{\phi_2^3} \cdots \]
be an inductive sequence of $\mathrm{C}^*$-algebras whose limit is $A$, such that the maps $\phi_i^j$ are injective and full.
Suppose $e_1$ is a full element in the Pedersen ideal of $A_1$, and set $e=\phi_1^\infty(e_1)$.
Then
\begin{equation}
\alabel{IndLimitROD-Eq}
 \rod(A,e) \leq \liminf \rod(A_i,\phi_1^i(e_i)).
\end{equation}
\end{prop}

\begin{proof}
Set $e_i := \phi_1^i(e_1)$ for all $i$.

Let $r := \liminf\, \rod(A_i,e_i)$, and suppose that $r' > r$.
Let $\eta = (r'-r)/2$.
Let $f \in \Aff(T_{e \mapsto 1}(A))_{++}$.
We note that $T_{e \mapsto 1}(A) = \varprojlim T_{e_i \mapsto 1}(A)$ (this is well-known in the unital case, and no tricks are needed to adapt the proof to the nonunital situation).
It therefore follows from Lemma \ref{AffDensity} (ii) that for all $i$ sufficiently large, there exists $g \in \Aff(T_{e_i \mapsto 1}(A_i))$ such that $\|g \circ \phi_\infty^i -f\| < \eta$.
In particular, we may find such $g \in \Aff(T_{e_i \mapsto 1}(A_i))$ for some $i$ for which $\rod(A_i,e_i) < r+\eta$.
This means that we may find $\hat a \in (A_i \tens \K)_+$ such that
\[ |d_\tau(\hat a) - g(\tau)| \leq r+\eta \]
for all $\tau \in T_{e_i \mapsto 1}(A_i)$.
Thus, with $a = \phi_i^\infty(\hat a) \in A$, we have
\[ |d_\tau(a) - f(\tau)| \leq |d_{\tau \circ \phi_\infty^i}(\hat a) - g(\tau)| + |g \circ \phi_i^\infty(\tau)-f(\tau)\| \leq r+\eta + \eta = r'. \]
Since $r' > r$ was arbitrary, this shows that $\rod(A) \leq r$ as required.
\end{proof}

\section{A computation of the Cuntz semigroup}
\alabel{CuCompSect}
In Theorem \ref{AUDcomputation}, we shall show that when $A$ is simple and exact, $\Cu(A)$ is almost unperforated, and the range of $\iota$ is uniformly dense, then $\Cu(A)$ can in fact be explicitly described purely in terms of the cone of traces $T(A)$ paired with the Murray-von Neumann semigroup $V(A)$.
Succinctly, we show that $\Cu(A)$ has the form described in \ccite{Corollary 6.8}{ElliottRobertSantiago}.
This sort of computation isn't particularly new --- such Cuntz semigroup computations were pioneered by Brown, Perera, and the second-named author in \cite{BrownPereraToms}, although with more hypotheses on $A$ including that it is unital.

The following preliminary is in order (again, slight weakenings of this can already be found in the literature).

\begin{lemma}\alabel{UnperfComp2}
Let $A$ be a simple $\mathrm{C}^*$-algebra.
Then $\Cu(A)$ is almost unperforated if and only if, for $[a],[b] \in \Cu(A)$, if $f([a]) < f([b])$ for every lower semicontinuous dimension function $f$ for which $f([b]) < \infty$ then $[a] \leq [b]$.
\end{lemma}

\begin{proof}
By \ccite{Proposition 3.2}{Rordam:Z}, $\Leftarrow$ is automatic.
Let us assume that $\Cu(A)$ is almost unperforated.
If $[c],[d] \in \Cu(A)$ are such that $f([c]) < f([d])$ holds for every lower semicontinuous dimension function, then we must show that $[c] \leq [d]$.
If we knew that $f([c]) < f([d])$ for the non-lower semicontinuous dimension functions, then \ccite{Proposition 3.2}{Rordam:Z} would show that $[c] \leq [d]$; the rest of the proof overcomes this obstacle.

Given any dimension function $f:\Cu(A) \to [0,\infty]$, we may define $\bar{f}:\Cu(A) \to [0,\infty]$ by
\[ \bar{f}([x]) = \sup_{[x'] \ll [x]} f([x']). \]
Then by \ccite{Lemma 4.7}{ElliottRobertSantiago}, $\bar{f}$ is a lower semicontinuous dimension function on $\Cu(A)$.

For $[a] \ll [c]$, we have
\[ f([a]) \leq \bar{f}([c]) < \bar{f}([d]) \leq f([d]). \]
(The first and last inequalities are evident from the definition of $\bar{f}$ while the middle one is by hypothesis.)
Therefore, by \ccite{Proposition 3.2}{Rordam:Z}, $[a] \leq [d]$.
But since $[c]$ is the supremum of $[a]$ satisfying $[a] \ll [c]$, we must have $[c] \leq [d]$, as required.
\end{proof}

In the following, we view $V(A) \dunion \Lsc(T(A),(0,\infty])$ as an ordered abelian semigroup as follows.
$V(A)$ and $\Lsc(T(A),(0,\infty])$ are already ordered semigroups (with pointwise $\leq$ giving the ordering on the latter), and we insist that their embeddings maintain the order and semigroup structures.
For $[p] \in V(A)$ and $f \in \Lsc(T(A),(0,\infty])$, we set $[p]+f:= g \in \Lsc(T(A),(0,\infty])$ given by $g(\tau) = \tau(p)+f(\tau)$; $[p] \leq f$ if and only if $\tau(p) < f(\tau)$ for all $\tau \in T(A)$, while $f \leq [p]$ if and only if $f(\tau) \leq \tau(p)$ for all $\tau \in T(A)$.

\begin{thm}
\alabel{AUDcomputation}
Let $A$ be a simple, exact $\mathrm{C}^*$-algebra such that $\Cu(A)$ is almost unperforated and the range of $\iota$ is uniformly dense.
Then $\Cu(A)$ is isomorphic, as an ordered semigroup, to $V(A) \dunion \Lsc(T(A),(0,\infty])$.
The isomorphism sends $[a] \in \Cu(A)$ to $[p] \in V(A)$ if $[a] = [p]$ for some projection $p \in (A \tens \K)_+$, and to the function $\iota(a)$ otherwise.
\end{thm}

\begin{proof}
The statement of the proposition implicitly defines a map
\[ \Phi:\Cu(A) \to V(A) \dunion \Lsc(T(A),(0,\infty]). \]
Let us first verify that $\Phi$ is an order embedding, i.e.\ that $[a] \leq [b]$ if and only if $\Phi([a]) \leq \Phi([b])$.
This will require only that $\Cu(A)$ is almost unperforated.
Four different cases need to be checked, depending on whether or not each of $[a],[b]$ is in $V(A)$.

It is trivial if both are in $V(A)$.
By using Proposition \ref{UnperfComp2}, we obtain the ``if'' direction when $[a] \in V(A)$.
However, if $[a] \in V(A)$ and $[a] < [b]$ then by \ccite{Proposition 2.2}{PereraToms:recasting}, there exists a nonzero $[x]$ such that $[a] + [x] \leq [b]$.
Since $A$ is simple, $d_\tau(x) > 0$ and so $\widehat{[a]}(\tau) < \widehat{[b]}(\tau)$ for all $\tau$.

The ``only if'' direction is automatic if $[a] \not\in V(A)$.
On the other hand, if $[a] \not\in V(A)$ and $\widehat{[a]} \leq \widehat{[b]}$ pointwise then, again using \ccite{Proposition 2.2}{PereraToms:recasting} and simplicity of $A$, we have
\[ \widehat{[a']}(\tau) < \widehat{[a]}(\tau) \leq \widehat{[b]}(\tau) \]
for all $\tau$, and therefore by Proposition \ref{UnperfComp2}, $[a'] \leq [b]$.
Since $[a]$ is the supremum of $[a'] \ll [a]$, we have $[a] \leq [b]$.
This concludes the verification that $\Phi$ is an order embedding.

Now we will show that $\Phi$ is surjective.
Obviously, $V(A)$ is in the range of $\Phi$.
We shall therefore show that $\Lsc(T(A),(0,\infty])$ is contained in the range of $\Phi$.
By Proposition \ref{ContinuousDenseSuffices}, it suffices to show that the range of $\iota$ contains $\Aff(T_{e \mapsto 1}(A), (0,\infty))$, as we shall now do.

Namely, given $f \in \Aff(T_{e \mapsto 1}(A), (0,\infty))$, let $\e_1 > \e_2 > \cdots$, converging to $0$, such that $f(\tau) > \e_1$ for all $n$.
Since the range of $\iota$ is dense, we may find $[b_n] \in \Cu(A)$ such that
\[ \iota(b_n)(\tau) \in (f(\tau)-\e_n, f(\tau)-\e_{n+1}) \]
for all $\tau \in T_{e \mapsto 1}(A)$.
By Proposition \ref{UnperfComp2}, $([b_n])$ is an increasing sequence, and its supremum $[b]$ clearly satisfies $\iota(b) = f$, as required.
\end{proof}

\section{Simple approximately subhomogeneous algebras with slow dimension growth}
\alabel{ASHSect}

The following result is in all likelihood true without the assumption of a compact primitive ideal space, but we do not require this generality for the applications we have in mind.  

\begin{prop}
\alabel{RankDimROD}
Let $A$ be subhomogeneous with $\Prim(A)$ compact, and let $e$ be a full, positive element in the Pedersen ideal of $A$.  It follows that
\begin{equation}
\alabel{RankDimROD-Ineq}
\rod(A,e) \leq 16 R_{(d+4):r}(e) := 16\sup_{\pi \text{ irred.\ rep.}} \frac{d_{top}(\pi)+4}{\rank \pi(e)} \\
\end{equation}
\end{prop}

\begin{proof}
This proof is contained in the proof of \ccite{Theorem 3.4}{Toms:rigidity}.  
We shall explain exactly how, since the statement of \ccite{Theorem 3.4}{Toms:rigidity} neither makes reference to the radius of divisibility, nor handles the nonunital case.
Set $r$ to be the right-hand side of \eqref{RankDimROD-Ineq}  for convenience, and let a continuous linear function $f:T(A) \to (0,\infty)$ be given.  By Proposition \ref{TChoquet} we have that $T_{e \mapsto 1}(A)$ is  a compact base for the space of densely defined lower semicontinuous traces.  We therefore need only prove that there exists $a \in (A \tens \K)_+$ such that
\[
 |d_\tau(a) - f(\tau)| \leq r, \ \forall \tau \in T_{e \mapsto 1}(A),
 \]
and, of course, it in fact suffices to show this only for extreme points of $T_{e \mapsto 1}(A)$.

As long as $r$ is finite, it follows from \ccite{Corollary 3.3}{ProjlessReg} that $A$ is a recursive subhomogeneous algebra, so let us consider it to be equipped with a recursive subhomogeneous decomposition. 
For $1 \leq i \leq n$, let $\mathrm{M}_{n_i}(X_i)$ be the $i^{\mathrm{th}}$ matrix block of $A$, $X_i^{(0)}$ the $i^{\mathrm{th}}$ clutching space, $A_i$ the $i^{\mathrm{th}}$ stage algebra, and $\phi_i:A_i \to \mathrm{M}_{n_i}$ the $i^{\mathrm{th}}$ clutching homomorphism.  The irreducible representations of $A$ correspond to evaluating an element of $A$ at a point
 \[
 x \in X^{(1)} := \bigcup_{i=1}^l X_i \backslash X_i^{(0)}.
 \]
 Let's denote such a representation by $\pi_x$, and let $\mathrm{Tr}$ denote the canonical normalized trace on $\mathrm{M}_n$ (for any $n$).   The extreme points of $T_{e \mapsto 1}(A)$ are all multiples of $\tau_x := \mathrm{Tr} \circ \pi_x$; therefore, we must verify that
 \[
 |d_{\tau_x}(a) - f(\tau)| \leq r d_{\tau_x}(e), \ \forall  x \in X^{(1)}.
 \]
 We note here that $d_{\tau_x}(b) = \mathrm{rank}(\pi_x(b))/n_i = \mathrm{rank}(b(x))/n_i$ for any $b \in A_+$ and $x \in X_i$.   We can therefore finally characterize our requirement in terms of ranks of positive operators:
 \[
 \left| \frac{\mathrm{rank}(a(x))}{n_i} - f(\tau_x) \right| \leq r \frac{\mathrm{rank}(e(x))}{n_i} = 16  \frac{\mathrm{dim}(X_i) + 4}{n_i}.
 \]
 The existence of such an $a$ now follows verbatim from the proof of Theorem 3.4 of \cite{Toms:rigidity}, beginning at the bottom of page 239.
\end{proof}

\begin{cor}
\alabel{SDGranks}
If $A$ is a simple, nonelementary, approximately subhomogeneous algebra with slow dimension growth then $\iota$ is surjective, and $\Cu(A)$ is as described in Theorem \ref{AUDcomputation}
\end{cor}

\begin{proof}
The range of $\iota$ is dense by Propositions \ref{IndLimitROD} and \ref{RankDimROD}, together with noting that simplicity and nonelementarity implies that $\min \rank \pi(\phi_1^\infty(a)) \to \infty$ for every nonzero $\pi$.
We have almost unperforation by \ccite{Corollary 5.9}{ProjlessReg}.
The conclusion follows from Theorem \ref{AUDcomputation}.
\end{proof}

\end{document}